\documentclass[12pt]{amsart}

\usepackage{amsmath,amssymb,latexsym}
\usepackage[a4paper,total={16cm,23.5cm},top=3cm,left=2.5cm]{geometry}
\usepackage{hyperref}

\def\im{\mathrm{im}}

\theoremstyle{plain}
\newtheorem*{mc}{Meta-Conjecture}
\newtheorem{theorem}{Theorem}[section]
\newtheorem{proposition}[theorem]{Proposition}
\newtheorem{fact}[theorem]{Fact}
\newtheorem{lemma}[theorem]{Lemma}
\newtheorem{corollary}[theorem]{Corollary}
\newtheorem{claim}{Claim}
\newtheorem*{clm}{Claim}

\theoremstyle{definition}
\newtheorem{definition}[theorem]{Definition}
\newtheorem{remark}[theorem]{Remark}
\newtheorem{question}[theorem]{Question}

\newcommand{\be}{\begin{enumerate}}
\newcommand{\bi}{\begin{itemize}}

\newcommand{\bd}{\begin{definition}}

\newcommand{\bt}{\begin{theorem}}
\newcommand{\bl}{\begin{lemma}}
\newcommand{\bc}{\begin{corollary}}
\newcommand{\bft}{\begin{fact}}
\newcommand{\bp}{\begin{proposition}}
\newcommand{\br}{\begin{remark}}
\newcommand{\er}{\end{remark}}
\newcommand{\ep}{\end{proposition}}
\newcommand{\ef}{\end{fact}}
\newcommand{\ec}{\end{corollary}}
\newcommand{\ee}{\end{enumerate}}
\newcommand{\ei}{\end{itemize}}
\newcommand{\ed}{\end{definition}}
\newcommand{\et}{\end{theorem}}
\newcommand{\el}{\end{lemma}}
\newcommand{\bpf}{\begin{proof}}
\newcommand{\bpfc}{\begin{proof}[Proof of Claim]}
\newcommand{\epf}{\end{proof}}

\def\Ind#1#2{#1\setbox0=\hbox{$#1x$}\kern\wd0\hbox to 0pt{\hss$#1\mid$\hss}
\lower.9\ht0\hbox to 0pt{\hss$#1\smile$\hss}\kern\wd0}

\def\Notind#1#2{#1\setbox0=\hbox{$#1x$}\kern\wd0\hbox to 0pt{\mathchardef
\nn="3236\hss$#1\nn$\kern1.4\wd0\hss}\hbox to 0pt{\hss$#1\mid$\hss}\lower.9\ht0
\hbox to 0pt{\hss$#1\smile$\hss}\kern\wd0}

\def\acl{\mathrm{acl}}
\def\dcl{\mathrm{dcl}}

\def\M{\mathfrak M}
\def\A{\mathfrak A}
\def\tp{\mbox{tp}}
\def\ale{\lesssim}

\def\aeq{\sim}
\def\ann{\mathrm{ann}}
\def\aann{\widetilde\ann}
\def\id{\mathrm{id}}
\def\dom{\mathrm{dom}\,}
\def\coker{\mathrm{coker}\,}
\def\tp{\mathrm{tp}}
\def\bdn{\mathrm{bdn}}
\def\rbdn{\mathrm{rbdn}}
\def\C{\tilde C}
\def\Z{\tilde Z}
\def\Q{\mathbb Q}

\begin{document}
\title{On $\omega$-categorical groups and rings of finite burden}
\author{Jan Dobrowolski}
\address{Instytut Matematyczny, Uniwersytetu Wroc\l awskiego, pl.\ Grunwaldzki 2/4, 50-383 Wroc\l aw\newline
\indent {\em and}\newline
\indent School of Mathematics, University of Leeds, Leeds LS2 9JT, UK}
\email{dobrowol@math.uni.wroc.pl}
\email{J.Dobrowolski@leeds.ac.uk}

\author{Frank O. Wagner}
\address{Universit\'e de Lyon; Universit\'e Claude Bernard Lyon 1; CNRS; Institut Camille Jordan UMR5208, 43 bd du 11 novembre 1918, 69622 Villeurbanne Cedex, France}

\email{wagner@math.univ-lyon1.fr}

\keywords{finite burden, $\omega$-categorical, group, ring, virtually abelian, virtually null}
\date{\today}
\subjclass[2010]{03C45}
\thanks{Partially supported by ANR-13-BS01-0006 ValCoMo, by
European Union's Horizon 2020 research
and innovation programme under the Marie Sklodowska-Curie grant agreement No 705410, and by the Foundation for 
Polish Science (FNP)}
\begin{abstract}An $\omega$-categorical group of finite burden is virtually finite-by-abelian; 
an $\omega$-categorical ring of finite burden is virtually finite-by-null; an $\omega$-categorical NTP$_2$ ring is nilpotent-by-finite.\end{abstract}
\maketitle

\section{Introduction}
A structure $\M$ is {\em $\omega$-categorical} if its theory has a unique countable model up to isomorphism. 
Basic examples include the pure set, the dense linear order, the random graph, and vector spaces over a finite field. A fundamental theorem by Ryll-Nardzewski
(proven independently by Svenonius and Engeler, \cite{RN59,S59,E59}) states that a structrue is $\omega$-categorical if and only if in any arity there are only finitely many parameter-free definable sets, up to equivalence.

There is a long history of study of $\omega$-categorical groups. In the general case, the 
main result is Wilson's classification of characteristically simple $\omega$-categorical groups as either elementary abelian, certain groups of functions from Cantor space to some finite simple group, or perfect $p$-groups (see Fact \ref{Wilson}); he conjectured that the third possibility is impossible (but this is still open). While a complete classification of all $\omega$-categorical groups
(and rings) appears out of reach at present, the question seems accessible under some model-theoretic tameness assumptions, giving rise to the following meta-conjecture:
\begin{mc}\begin{enumerate}\item A tame $\omega$-categorical group or ring is virtually nilpotent.
\item A supertame $\omega$-categorical group is virtually finite-by-abelian; a supertame $\omega$-categorical ring is virtually finite-by-null.\end{enumerate}\end{mc}
(Recall that a group/ring is {\em virtually $P$} if it has a finite index subgroup/-ring which is $P$; it is {\em finite-by-$P$} if it has a normal subgroup/ideal $I$ such that it is $P$ moulo $I$; moreover a ring is {\em null} if multiplication is trivial.) Of course, one has to specify the precise meaning of tame.

We shall prove a general theorem about $\omega$-categorical bilinear quasi-forms of {\em finite burden}, and deduce Conjecture (2) in the finite burden case; moreover, we show (1) for rings with NTP$_2$. Here, NTP$_2$ is a combinatorially defined very general model-theoretic tameness condition currently under intense investigation in neostability theory, and burden, also called {\em inp-rank}, is a cardinal-valued rank well defined (i.e.\ not assuming value $\infty$) precisely on the class of NTP$_2$ theories,  thus providing a hierarchy inside of this class (see Definition \ref{inp_def}). The principal examples of structures of burden $1$ are real closed fields (and expansions thereof with (weakly) $o$-minimal theories),
the valued fields of $p$-adic numbers for any prime $p$, valued algebraically closed fields, Presburger arithmetic 
$(\mathbb{Z}, 0,+,<)$, as well as the random graph (and any other weight one simple theory, by \cite[Proposition 8]{Ad}).
By sub-multiplicativity of burden \cite[Theorem 2.5]{Ch}, finite burden structures include all structures interpretable in 
inp-minimal ones, e.g.\ algebraic groups over the fields of real, complex and $p$-adic numbers.
For more details on burden and related topics see \cite{Ch} or \cite{Ad}.

\subsubsection*{History of results} If tame is read as {\em stable}, then (1) has been shown for groups by Felgner \cite{Fe78} and for rings by Baldwin and Rose \cite{BR77}; (2) has been shown by Baur, Cherlin and Macintyre \cite{BCM79}; if tame means {\em simple}, then the group case of (2) has been shown by Evans and Wagner \cite{EW00}, and if tame means {\em NSOP} 
(so, in particular, if it means simple), the group case of (1) has been shown by Macpherson \cite{M88}. Finally, if tame is taken as {\em dependent}, then (1) has been shown by Krupinski \cite{Kr}, assuming in addition {\em finitely satisfyable generics} for the group case. Moreover, building on work of Baginski \cite{Ba09}, he proves that the versions of (2) for {\em nilpotent} groups and for rings are equivalent \cite{Kr11} (in fact he does not explicitly cover the case with finite normal subgroups/ideals, but his proof adapts), and the group version of (1) implies that for rings. In particular, (1) also holds for NSOP rings. Finally, Kaplan, Levi and Simon \cite{KLS} show (1) for dependent groups of burden $1$. 
Note that extraspecial $p$-groups \cite{Fe} yield an example showing that the finite normal subgroup cannot be avoided in (2), unless one assumes the existence of connected components (which holds, for instance, in dependent theories).

An earlier version of this paper \cite{DW} obtained the same results under the stronger hypothesis of burden $1$; virtually the only consequence used was that any two definable groups are comparable with respect to almost inclusion. 
For the generalisation to the finite burden case, we use essentially the same proof; considerable work is being spent to show that all the relevant groups are still comparable with respect to almost inclusion in a minimal counterexample of finite burden.

The paper is organized as follows: In Section \ref{inp}, we recall the definition of burden, and  deduce some algebraic consequences when the burden is finite.  In section \ref{addrel}, we introduce additive relations and the ring of quasi-endomorphisms; in Section \ref{quasi}, we study the properties of quasi-homomorphisms under the assumption of $\omega$-categoricity. In Section \ref{bil}, we generalize the notion of a bilinear form using quasi-homomorphisms instead of homomorphisms.
In section \ref{vat}, we
prove our Main Theorem, Theorem \ref{trivial_by_finite}, about virtual almost triviality of bilinear quasi-forms. This is applied in Section \ref{main_results} to obtain the results about groups and rings. 
In Section \ref{q+r}, we state some questions and we prove that $\omega$-categorical
rings with NTP$_2$ are virtually nilpotent.

We would like to thank the anonymous referee for his careful reading, and for pointing out a missing assumption in what is now Lemma \ref{epi}.

\section{Burden}\label{inp}
Throughout the paper we will work in a monster model of the relevant complete theory 
(i.e.\ a $\bar{\kappa}$-saturated, $\bar{\kappa}$-homogeneous model, where $\bar{\kappa}$ is a sufficiently big cardinal number). Definability of a set is with parameters, and includes imaginary sets, i.e.\ definable sets modulo definable equivalence relations (as we shall want to talk about the quotient of a definable group by a definable normal subgroup). For the basic notions of model theory, the reader may want to consult\cite{H93},  \cite{Po} or \cite{TZ}.

\begin{definition}\label{inp_def}
 \begin{enumerate}
\item Let $\kappa$ be a cardinal number. An \emph{inp-pattern of depth $\kappa$ in a partial type
$\pi(\overline{x})$} is a sequence 
$\langle \varphi_i(\overline{x}, \overline{y}_i) : i < \kappa \rangle$ of formulas and an array 
$\langle \overline{a}_{i,j} : i < \kappa, j < \omega \rangle$ of parameters such that:
\begin{enumerate}
\item For each $i < \kappa$, there is some $k_i < \omega$ such that
$\{\varphi_{i}(\overline{x}, \overline{a}_{i, j}) : j < \omega \}$ is $k_i$-inconsistent; and
\item For each $\eta : \kappa \rightarrow \omega$, the partial type 
$$\pi(\overline{x}) \cup \{\varphi_i(\overline{x}, \overline{a}_{i, \eta(i)}) : i < \kappa \}$$ is consistent.
\end{enumerate}
\item The \emph{burden} (or \emph{inp-rank}) of a partial type $\pi(\overline{x})$ is the maximal 
$\kappa$ such that there is an inp-pattern of depth $\kappa$ in $\pi(\overline{x})$, if such a maximum 
exists. In case there are inp-patterns of depth $\lambda$ in $\pi(\overline{x})$ for every cardinal 
$\lambda < \kappa$ but no inp-pattern of depth $\kappa$, we say that the burden of 
$\pi(\overline{x})$ is $\kappa_-$. We will denote the burden of $\pi(\overline{x})$ by 
$\bdn(\pi(\overline{x}))$. By the burden of a type-definable set we mean the burden of a type
defining this set (this, of course, does not depend on the choice of the type). A theory $T$ is called \emph{strong}, if the burden of any partial type in finitely many variables is bounded by $(\aleph_0)_{-}$.
\end{enumerate}
\end{definition}
Note that the formulas $\varphi_i$ can be taken parameter-free, as we may incorporate eventual parameters into the $\bar a_{i,j}$. Clearly, burden does not depend on the base parameters.
\begin{remark}\label{bdn_product}
 Suppose $k=\bdn(\pi(\overline{x}))$  and $l=\bdn(\rho(\overline{y}))$ are finite, where $\overline{x}$
 and $\overline{y}$ are disjoint. Then $\bdn(\pi(\overline{x})\cup\rho(\overline{y}))\geq k+l$.
 In other words, for type-definable sets $V$ and $W$ of finite burden we have:
 $\bdn(V\times W)\geq \bdn(V)+\bdn(W)$.
\end{remark}
\begin{proof}
This is clear, as the concatenation of an inp-pattern in $\pi(\overline{x})$ with an
inp-pattern in $\rho(\overline{y})$ is an inp-pattern in 
$\pi(\overline{x})\cup\rho(\overline{y})$.
\end{proof}

\begin{remark}\label{bdn_cover}
Suppose  $f:V\to W$ is definable and all fibres of $f$ have size at most $k$, where $k<\omega$. Then
 $\bdn(V)\leq \bdn(W)$.
\end{remark}
\begin{proof}
Suppose 
$\langle \varphi_i(v, y_i) : i < \kappa \rangle$ together with
$\langle {a}_{ij} : i < \kappa, j < \omega \rangle$ form an inp-pattern in $V$.
We may assume that $\langle {a}_{ij}:j<\omega \rangle$ are pairwise distinct for any $i<\kappa$.
Put $\psi_i(w,y):=(\exists v)(\varphi_i(v, y_i)\wedge f(v)=w)$ for $ i < \kappa$. We claim that these form an inp-pattern in $W$ (with the same parameters). Indeed, for any $i<\kappa$, if $\ell_i$ is such that
$\{\varphi_{i}(v, a_{i, j}) : j < \omega \}$ is $\ell_i$-inconsistent, then by the pigeonhole principle
$\{\psi_{i}(w, a_{i,j}) : j < \omega \}$ is $(\ell_i-1)k+1$ inconsistent. Also, for each
$\eta : \kappa \rightarrow \omega$, if $v_0\in V$ satisfies $\varphi_{i}(v, a_{i, \eta(i)})$ for each $i<\kappa$,
then $f(v_0)\in W$ satisfies $\psi(w, a_{i,\eta(j)})$ for each $i\in \kappa$.
\end{proof}
 
For the next results we introduce some notation for subgroups $H$ and $K$ of a group $G$. We say that $H$ is {\em almost contained} in 
$K$, denoted $H\ale K$, if $H\cap K$ has finite index in $H$. If $H\ale K$ and $K\ale H$, 
the two groups are {\em commensurable}, denoted $H\aeq K$. The {\em almost centraliser} of $H$ is defined as 
$$\C_G(H)=\{g\in G:H\ale C_H(g)\},$$
and the {\em almost centre} of $G$ is $\Z(G)=\C_G(G)$.

The following fact is a special case of \cite[Theorem 2.10]{Hempel}. Recall that the ambient model should be sufficiently saturated. So we cannot just add predicates for $H$ and $K$.
\bft\label{Csymmetry} If $H$ and $K$ are definable, then $H\ale\C_G(K)$ if and only if $K\ale\C_G(H)$.\ef
We now turn to the consequences of finite burden we use.
\bft[{\cite[Corollary 2.3]{CKS}}]\label{reducible-cap} Let $G$ be an abelian group with NTP$_2$ and $\langle H_i:i\in I\rangle$ a family of uniformly definable subgroups. Then there is $n$ such that for all $I_0\subseteq I$ of size at least $n$ there is $i_0\in I_0$ with $\bigcap_{i\in I_0\setminus\{i_0\}}H_i\ale H_{i_0}$.\ef
Thus any irreducible intersection $\bigcap_{i<n}H_i$ (meaning that $\bigcap_{j\not=i}H_j\not\ale H_i$ for all $i<n$) of uniformly definable groups has its size $n$ bounded as a function of the formula used to define the $H_i$.
\bl\label{reducible} Let $G$ be an abelian group of finite burden, and $\langle H_i:i<n\rangle$ definable subgroups of $G$.
If the sum $\sum_{i<n}H_i$ is irreducible (meaning that $H_i\not\ale\sum_{j\not=i}H_j$ for all $i<n$), then $n\le\bdn(G)$.\el
\bpf Let $\varphi_i(x,y)$ be the formula $x-y\in\sum_{j\not=i} H_j$, and choose $\langle a_{i,j}:j<\omega\rangle$ to be representatives in $H_i$ for distinct cosets of $\sum_{j\not=i} H_j$. Then $\langle\varphi(x,a_{i,j}):j<\omega\rangle$ is $2$-inconsistent, and consistency of any path $\sigma\in\omega^n$ is witnessed by $\sum_{i<n}a_{i,\sigma(i)}$. So we obtain an inp-pattern of depth $n$.\epf

\section{Additive relations and quasi-endomorphisms}\label{addrel}
We extend the construction of the definable quasi-endomorphisms ring from \cite[Section 3.2]{BCM79} to non-connected groups.

\bd\label{d:additive}  Let $G$ and $H$ be abelian groups. An {\em additive relation} between $G$ and $H$ is a subgroup $R\le G\times H$. We call $\pi_1(R)$, the projection to the first coordinate, the {\em domain} $\dom R$ and $\pi_2(R)$ the image $\im R$ of $R$; the subgroup $\{g\in G:(g,0)\in R\}$ is the {\em kernel} $\ker R$, and $\{h\in H:(0,h)\in R\}$ is the cokernel $\coker R$. If $\dom R$ has finite index in $G$ and $\coker R$ is finite, the additive relation $R$ is a {\em quasi-homomorphism} from $G$ to $H$ 
(not to be confused with quasi-homomorphism in the sense of metric groups). A quasi-homomorphism $R$ induces a homomorphism $\dom R\to H/\coker R$. If $G=H$ we call $R$ a {\em quasi-endomorphism}. Particular additive relations are $\id_G=\{(g,g):g\in G\}$ and $0_G=G\times\{0\}$.\ed
\br\label{finin=finim} Let $g\le G\times H$ be a quasi-homomorphism. Then $|G:\ker g|$ is finite if and only if $\im\,g$ is finite.\er
\bpf This is trivial for the induced homomorphism from $\dom g$ to $H/\coker g$. The result follows.\epf

\bd\bi\item
If $R\le G\times H$ is an additive relation, $g\in G$ and $K\le G$, put $R(g)=\{h\in H:(g,h)\in R\}$ and $R[K]=\bigcup_{g\in K}R(g)$.
\item If $R,R'\le G\times H$ are additive relations, put
$$R+R'=\{(a,b+b')\in G\times H:(a,b)\in R,\,(a,b')\in R'\}.$$
This is again an additive relation. If moreover $R$ and $R'$ are quasi-homomorphisms from $G$ to $H$, so is $R+ R'$. Note that $R+R'$ (as additive relations) is different from the sum when $R$ and $R'$ are considered as subgroups. Similarly for $R-R'=\{(a,b-b'):(a,b)\in R,\,(a,b')\in R'\}$.
\item
We call $R,R'\le G\times H$ {\em equivalent}, denoted $R\equiv R'$, if there is a
subgroup $G_1$ of finite index in $G$ and a finite group $F\le H$ such that
$$(R\cap(G_1\times H))+(G_1\times F)=(R'\cap(G_1\times H))+(G_1\times F).$$
This is clearly an equivalence relation.
\item
If $R\le G\times H$ and $R'\le H\times K$ are additive relations, put
$$R'\circ R=\{(a,c)\in G\times K:\exists b\,[(a,b)\in R\mbox{ and }(b,c)\in R']\}.$$
This is again an additive relation between $G$ and $K$. If $R$ and $R'$ are quasi-homomorphisms, so is $R'\circ R$. 
We denote the $n$-fold composition of $R$ with itself by $R^{\circ n}$.
\item
For an additive relation $R\le G\times H$ put
$$R^{-1}=\{(h,g)\in H\times G:(g,h)\in R\}.$$
Note that this is also an additive relation between $H$ and $G$.\ei\ed
\br\label{inverse} Note that 
$$R^{-1}\circ R=\id_{\dom R}+(\dom R\times\ker R)\quad\mbox{and}\quad R\circ R^{-1}=\id_{\im\,R}+(\im\,R\times\coker R).$$
If $\im\, R$ has finite index in $H$ and $\ker R$ is finite, then $R^{-1}$ is a
quasi-homomorphism from $H$ to $G$. If moreover $R$ is a quasi-homomorphism,
then $R\circ R^{-1}\equiv\id_H$ and $R^{-1}\circ R\equiv\id_G$.\er
By \cite[Lemma 27]{BCM79} addition is associative and commutative, multiplication is associative, $G\times\{0\}$ is an additive and the diagonal $\{(g,g):g\in G\}$ is a multiplicative identity.
\bl\label{endom} Let $G$ be an abelian group. The sum, difference and product of definable quasi-endomorphisms of $G$ is again a definable quasi-endomorphism. The set of definable quasi-endomorphisms of $G$ modulo equivalence forms an associative ring.\el
\bpf The proofs of \cite[Lemmas 29, 31 and 32]{BCM79} carry over {\em verbatim}, for the equivalence $R\doteq R'$ if there is finite $F$ with $R+(G\times F)=R'+(G\times F)$ (which is finer than our equivalence $\equiv$). In particular addition and subtraction are well-defined modulo $\doteq$ and $R-R$ is $\doteq$-equivalent to zero. Moreover, for quasi-endomorphisms $R$, $S$ and $T$ of $G$ we have\begin{enumerate}
\item $R(S+T)\subseteq RS+RT\subseteq (RS+RT)+(G\times F)$ for some finite $F\le G$.
\item $SR+TR\subseteq (S+T)R\subseteq (SR+TR)+(G\times F)$ for some finite $F\le G$.
\item Multiplication is well-defined modulo $\doteq$.\end{enumerate}
Thus the quasi-endomorphisms of $G$ modulo $\doteq$ form a ring.

However, $R\equiv R'$ if and only if there is $G_1$ of finite index in $G$ with $R\cap(G_1\times G_1)\doteq R'\cap(G_1\times G_1)$. Moreover, definability is preserved under sum, difference and product of quasi-endomorphisms. The result follows.
\epf

\section{Quasi-homomorphisms of $\omega$-categorical groups}\label{quasi}
Recall that a complete first order theory in a countable language is said to be $\omega$-categorical if
it has only one countable model up to isomorphism, and a structure $M$ is $\omega$-categorical
if $Th(M)$ is. By the Ryll-Nardzewski Theorem, this is equivalent to the following statement:
for every $n<\omega$ there are only finitely many complete $n$-types over $\emptyset$.
Hence, for any finite set $A$ in an $\omega$-categorical structure $M$
there are only finitely many definable sets over $A$, and $\omega$-categorical 
structures are uniformly locally finite (i.e. there if a function $f:\omega\to \omega$ such
that, for any $n\in \omega$, each substructure of $M$ generated by $n$ elements has at most $f(n)$ elements) \cite[Corollary 7.3.2]{H93}.
\begin{lemma}\label{add_basics}
 Let $G$ and $H$ be abelian groups, and let $g\leq G\times H$ be an additive relation. \be
\item If $\coker g$ is finite, $|H:\im\,g|$ is finite, and $H_1 \leq H$ has infinite index in $H$, then 
 $|\dom g:g^{-1}[H_1]|$ is infinite.
\item If $\ker g$ is finite, $|G:\dom g|$ is finite, and $G_1\leq G$ has infinite index, then $|\im\,g:g[G_1]|$ is infinite.
\item If $H_1 \leq H$, then $|\dom g:g^{-1}[H_1]|\le|\im\,g:\im\,g\cap H_1|$.\ee
\end{lemma}
\begin{proof}\be
\item Let $\langle h_i:i<\omega\rangle$ be such that $h_i-h_j\notin H_1+\coker g$ for
$i\neq j$. Since $|H:\im\,g|$ is finite, we may assume that all $g_i$ are in 
the same coset of $\im g$,  so without loss of generality they are all in $\im g$. For each $i$ let $g_i\in G$ be such that $h_i\in g(g_i)$. If 
 $g_i-g_j\in g^{-1}[H_1]$ for $i\neq j$, then there is $h\in H_1$ such that $h\in g(g_i-g_j)$, so $h-(h_i-h_j)\in\coker g$, a contradiction.
 Hence all $g_i$ are in pairwise distinct cosets modulo  $g^{-1}[H_1]$.
\item Follows from (1) applied to $g^{-1}$.
\item If elements $\langle g_i:i\in I\rangle$ are pairwise distinct modulo $g^{-1}[H_1]$
elements in $\dom g$, and $h_i\in g(g_i)$, then the elements $\langle h_i:i\in I\rangle$ are in pairwise distinct cosets modulo~$H_1$.\qedhere\ee
\end{proof}

\bl\label{epi}Let $G$ and $H$ be definable abelian groups in an $\omega$-categorical structure, and $f,g\le G\times H$ definable additive relations such that $\ker f$ and $\coker g$ are finite, $\im\,g$ has finite index in $H$, and $\dom f$ has finite index in $G$. Then $\ker g$ and $\coker f$ 
are finite, $\im f$ has finite index in $H$ and $\dom g$ has finite index in $G$.\el
\bpf Let $A$ be a finite set over which all the above objects are definable. 
\begin{clm} Suppose that $H_1<H_2\le H$ are such that $H_1$ has infinite index in $H_2$.
Then $f[g^{-1}[H_1]]$ has infinite index in $f[g^{-1}[H_2]]$.\end{clm}
\bpf As $\im\,g$ has finite index in $H$, the index of $H_1\cap\im\,g$ in $H_2\cap\im\,g$ is infinite. 
Now, $g^{-1}[H_1]$ has infinite index in $g^{-1}[H_2]$ by Lemma \ref{add_basics}(1) applied to $g\cap (g^{-1}[H_2]\times H_2)$,
so  $f[g^{-1}[H_1]]$ is a subgroup of infinite index in $f[g^{-1}[H_2]]$ by Lemma \ref{add_basics}(2) applied to $f\cap(g^{-1}[H_2]\times f[g^{-1}[H_2]])$.\epf

Suppose for a contradiction that $\ker g$ or $\coker f$ is infinite. Put $K_0=\{0\}\leq H$ and define inductively $K_{n+1}=f[g^{-1}[K_{n}]]$. Then $K_1$ is infinite; by the claim $K_n$ is a subgroup of infinite index in $K_{n+1}$ for all $n<\omega$,
contradicting $\omega$-categoricity (as this implies there are infinitely many disjoint definable sets $\langle K_{n+1}\setminus K_n:n<\omega\rangle$ over $A$ in $H$). 

Now suppose that $\im f$ has infinite index in $H$ or $\dom g$ has infinite index in $G$. 
Put $K_0=H$ and define as before $K_{n+1}=f[g^{-1}[K_{n}]]$. Then $K_1$ has infinite index in $K_0$; by the claim $K_{n+1}$ is a subgroup of infinite index in $K_{n}$ for all $n<\omega$, again contradicting $\omega$-categoricity.
\epf
\br Note that commutativity was not used in the proof. An analogous lemma holds for arbitrary groups, and {\em multiplicative} relations (with the obvious definition adapting Definition \ref{d:additive} to non-commutative groups).\er

\bl\label{kerim} Let $G$ and $H$ be abelian groups definable in an $\omega$-categorical structure, and $f,g\le G\times H$ definable quasi-homomorphisms. If $\ker f\ale \ker g$ and $\im\,f\ale\im\,g$, then $\im\,g\aeq \im f$ and $\ker g\aeq\ker f$.\el
\bpf 
Suppose $\ker f\ale\ker g$ and $\im\,f\ale\im\,g$.
Let $f_1,g_1\le G/(\ker f\cap\ker g)\times\im\,g$ be the additive relations induced by $f$ and $g$, namely
$$f_1(x+(\ker f\cap\ker g),y)\iff f(x',y)\mbox{ for some/all }x'\in x+(\ker f\cap\ker g),$$ and likewise for $g_1$.
Then $\ker f_1$ is finite since $\ker f\ale\ker g$, and $\coker g_1=g[\ker f\cap\ker g]=\coker g$ is finite, too.

Now $\im f\cap\im g$ has finite index in $\im f$, so $ f^{-1}[\im\,g\cap\im f]$ has
 finite index in $G$ by Lemma \ref{add_basics}(3); it follows that $\dom f_1= f^{-1}[\im\,g\cap\im f]/(\ker f\cap\ker g)$ has finite index in $G/(\ker f\cap\ker g)$. Moreover $\im\,g_1=\im\,g$. 
Thus $\im f_1=\im f\cap\im\,g$ has finite index in $\im\,g$ and $\ker g_1=\ker g/(\ker f\cap\ker g)$ is finite by Lemma \ref{epi}. Thus $\im f\aeq \im\,g$ and $\ker f\aeq\ker g$.\epf

\begin{corollary}\label{invertible}
Let $G$ and  $H$ be abelian groups definable in an  $\omega$-categorical theory,  $f\leq G\times G$  a definable quasi-endomorphism of $G$, and $g\leq G\times H$ a definable quasi-homomorphism.\be
\item $\ker f$ is finite if and only if $|G:\im f|$ is finite.
\item If  $G\leq H$ and $|H:\im\,g|$ is finite, then  $|H:G|$ and $\ker g$ are finite.\ee
\end{corollary}
\begin{proof}
 For (1), apply  Lemma \ref{kerim} to $f$ and $id_G$ for the implication, and to $id_G$ and $f$ for the converse. For (2) consider  the inclusion $i\leq G\times H$. As the assumptions imply $\im\,i\ale\im\,g$ we may apply Lemma \ref{kerim} and obtain $H\ale\im\,g\aeq\im\,i=G$ and $\ker g\aeq\ker i=\{0\}$. 
\end{proof}

\bl\label{sum}
Let $G$ be an $\omega$-categorical abelian group and $f$ a definable quasi-endomorphism of $G$. Then there is 
$n<\omega$ such that $G$ decomposes as an almost direct sum 
of $\im f^{\circ n}$ and $\ker f^{\circ n}$ (i.e. $G\aeq\im f^{\circ n}+\ker f^{\circ n}$ and 
$im f^{\circ n}\cap \ker f^{\circ n}$ is finite).\el
\bpf The $f^{\circ n}[G]$ form a descending chain of subgroups, all definable over the same finite set of parameters. 
By $\omega$-categoricity there is some $n$ such that $f^{\circ n}[G]=f^{\circ n+1}[G]=f^{\circ 2n}[G]$. 
Consider $g\in\dom f^{\circ n}$. There is $h\in f^{\circ n}[G]$ such that 
$f^{\circ n}(g)\cap f^{\circ n}(h)\not=\emptyset$. But this means $g-h\in\ker f^{\circ n}$, so 
$$G\ale \dom f^{\circ n}\le\im f^{\circ n}+\ker f^{\circ n}.$$
As  $f^{\circ n}[\im f^{\circ n}]=\im f^{\circ 2n}=\im f^{\circ n}$, the intersection 
$\im f^{\circ n}\cap\ker f^{\circ n}$ must be finite by applying Corollary \ref{invertible}(2) to 
$\im f^{\circ n}\le\im f^{\circ n}$ and $g=f^{\circ n}$.\epf

\section{Bilinear quasi-forms}\label{bil} We shall now introduce a generalization of the notion of a bilinear form. As before, definability will be with parameters in a monster model.
\bd Let $G$, $H$ and $K$ be abelian groups. A {\em bilinear quasi-form} is a partial function 
$\lambda:G\times H\to K$ such that for every $g\in G$ and $h\in H$  the partial functions
$\lambda_g:H\to K$ given by $x\mapsto\lambda(g,x)$ and $\lambda'_h:G\to K$ given by $\lambda'_h(y)=\lambda(y,h)$ are quasi-homomorphisms with trivial cokernel (i.e.\ partial homomorphisms defined on a subgroup of finite index).\ed
We shall call $\lambda$ {\em definable} if $G$, $H$, $K$ and $\lambda$ are definable.
\bd Let $\lambda:G\times H\to K$ be a bilinear quasi-form. For $g\in G$ (or $h\in H$) the {\em annihilator} of $g$ (or of $h$) is the subgroup
$$\begin{aligned}\ann_H(g)&=\{h\in H:\lambda(g,h)=0\}=\ker\lambda_g\le G\,\quad\mbox{or}\\
\ann_G(h)&=\{g\in G:\lambda(g,h)=0\}=\ker\lambda'_h\le H.\end{aligned}$$\ed
\br Of course the annihilators depend on the bilinear quasi-form; if it is not obvious from the context, we shall indicate this by a superscript: $\ann^\lambda$. \er
Suppose $\lambda:G\times H\to K$ is a bilinear quasi-form.
For any $g,g'\in G$, we shall consider the additive relation 
$\lambda_{g,g'}=\lambda_{g'}^{-1}\circ\lambda_g\le H\times H$ given by
$\{(h,h')\in H\times H:\lambda(g,h)=\lambda(g',h')\}$. Clearly
$\ker\lambda_{g,g'}=\ann_H(g)$ and $\coker\lambda_{g,g'}=\ann_H(g')$.
\bl\label{quasi-end} If $\ann_H(g')\ale\ann_H(g)$ and $\im\lambda_g\ale\im\lambda_{g'}$, then $\lambda_{g,g'}$ induces a quasi-endo\-morph\-ism $\bar{\lambda}_{g,g'}$ of $H/\ann_H(g')$ given by 
$$\bar{\lambda}_{g,g'}(x+\ann_H(g'),y+\ann_H(g'))\quad\Leftrightarrow\quad
\lambda_{g,g'}(x',y)\mbox{ for some }x'\in x+\ann_H(g').$$\el
\bpf Note first that this does not depend on the choice of $y$ in the coset $y+\ann_H(g')$, as $\ann_H(g')=\coker\lambda_{g,g'}$. 
Second, $\dom\lambda_g$ has finite index in $H$, so 
$$\dom\bar{\lambda}_{g,g'}=\dom\lambda_{g,g'}/\ann_H(g')=
\lambda_g^{-1}[\im\lambda_g\cap\im\lambda_{g'}]/\ann_H(g')$$
has finite index in $H/\ann_H(g')$ by Lemma \ref{add_basics}(3). Third, 
$$\begin{aligned}\coker\bar{\lambda}_{g,g'}&=\{h\in H: \exists\ h'\in\ann_H(g')\ \lambda(g,h')=\lambda(g',h)\}/\ann_H(g')\\
&=\lambda_{g'}^{-1}[\lambda_g[\ann_H(g')]]/\ann_H(g')\end{aligned}$$ is finite,
as $\lambda_g[\ann_H(g')]$ is finite due to $\ann_H(g')\ale\ann_H(g)$. So $\bar\lambda_{g,g'}$ is indeed a quasi-endomorphism.\epf
\bd For $A\le G$ and $B\le H$ put
$$\aann_H(A)=\{h\in H:A\ale\ann_G(h)\}\quad\mbox{and}\quad
\aann_G(B)=\{g\in G:B\ale\ann_H(g)\},$$
the {\em almost} annihilators of $A$ and $B$.\ed
As for the annihilators, the almost annihilators depend on the bilinear quasi-form $\lambda$, which will be indicated as a superscript if needed: $\aann^\lambda$.
\br We have $$\aann_H(A)=\{h\in H:\lambda'_h[A]\mbox{ is finite}\}\quad\mbox{and}\quad
\aann_G(B)=\{g\in G:\lambda_g[B]\mbox{ is finite}\}.$$\er
\bpf This follows from Remark \ref{finin=finim}.\epf

If $A\ale\ann_H(g)$ and $A\ale\ann_H(g')$ then $A\ale\ann_H(g)\cap\ann_H(g')\le\ann_H(g\pm g')$ (and symmetrically), so the almost annihilators are subgroups of $G$ and of $H$. Moreover, if $G$, $H$, $\lambda$, $A$ and $B$ are definable, they are given as a countable increasing union of sets definable over the same parameters, and will thus be definable in an $\omega$-categorical theory.

The next proposition is an adaptation of \cite[Theorem 2.10]{Hempel} to bilinear quasi-forms.
\bp\label{symmetry} Let $\lambda:G\times H\to K$ be a definable bilinear quasi-form,
and $A\le G$ and $B\le H$ be definable subgroups. Then $B\ale\aann_H(A)$ if and only
if $A\ale\aann_G(B)$.\ep
\bpf We may assume that $G$, $H$, $A$ and $B$ are defined over $\emptyset$. Suppose 
that $B\not\ale\aann_H(A)$. Consider a sequence $\langle h_i:i<\omega\rangle $ in $B$ representing 
different cosets of $\aann_H(A)$. Then $h_i-h_j\notin\aann_H(A)$ for $i\not=j$, so the 
index $|A:\ann_A(h_i-h_j)|$ is infinite. By Neumann's Lemma (\cite{BHN}) no finite union of 
cosets of the various $\ann_A(h_i-h_j)$ can cover $A$. By compactness and sufficient saturation
of the monster model, there is an infinite sequence $\langle g_k:k<\omega\rangle $ in $A$ such that 
$\lambda(g_k-g_\ell,h_i-h_j)\not=0$ for all $i\not=j$ and $k\not=\ell$.
It follows that $|B:\ann_B(g_k-g_\ell)|$ is infinite, whence  $g_k-g_\ell\notin\aann_G(B)$  for all $k\not=\ell$. Thus $A\not\ale\aann_G(B)$. 

The other direction follows by symmetry.\epf

\bd A bilinear quasi-form $\lambda$ is {\em almost trivial} if there is a finite subgroup of $K$ containing $\im\lambda$. It is {\em virtually almost trivial} if there are subgroups $G_0$ of finite index in $G$ and $H_0$ in $H$ such that the restriction of $\lambda$ to $G_0\times H_0$ is almost trivial.\ed

\bp\label{almost_trivial} Let $\lambda:G\times H\to K$ be a definable bilinear quasi-form. Then $\lambda$ is almost trivial if and only if there is a finite 
bound on the indices of $\ann_H(g)$ and $\ann_G(h)$ in $H$ and $G$, respectively, for all $g\in G$ and $h\in H$.\ep
\bpf Suppose $\im\lambda$ generates a finite group $K_0$. Since $\lambda$ is a bilinear
quasi-form, by compactness, there is a finite bound on the indices of $\dom\lambda_g$ in $H$ 
and of $\dom\lambda'_h$ in $G$. As the indices $|\dom\lambda_g:\ann_H(g)|$ and $|\dom\lambda'_h:\ann_G(h)|$ are bounded by $|K_0|$, the implication follows.

Conversely, suppose there is a finite bound $\ell$ for the indices of $\ann_H(g)$ and $\ann_G(h)$ in $H$ and in $G$ for all $g\in G$ and $h\in H$. Note that this implies that $\langle\im\lambda\rangle$ has finite exponent, as $n\,\lambda(g,h)=\lambda(g,nh)\in\lambda_g[H]$ for all $g\in G$, $h\in H$ and $n<\omega$. So it is enough to show that $\im\lambda$ is finite. 

Now $\ell$ also bounds the size of $\lambda_g[H]$ and of $\lambda'_h[G]$, for all $g\in G$ and $h\in H$, and we may consider $g\in G$ with $\lambda_g[H]$ maximal, and choose $h_0,\ldots,h_n\in H$ with
$\lambda_g[H]=\{\lambda(g,h_i):i\le n\}$. Then for $g'\in g+\bigcap_{i\le n}\ann_G(h_i)$ we have
$\lambda(g',h_i)=\lambda(g,h_i)$, whence $\lambda_{g'}[H]\supseteq\lambda_g[H]$, and $\lambda_{g'}[H]=\lambda_g[H]$ by maximality. 
Note that $\bigcap_{i\le n}\ann_G(h_i)$ is a subgroup of boundedly finite index in $G$ (i.e.\ bounded independently from
$g$). It follows that there can only be finitely many maximal sets of the form $\lambda_g[H]$ for $g\in G$, and $\im\lambda$ is finite.\epf

\bc\label{aann} Let $\lambda:G\times H\to K$ be a definable bilinear quasi-form.
The following are equivalent:\begin{enumerate}
\item $G\ale\aann_G(H)$.
\item $H\ale\aann_H(G)$.
\item $\lambda$ is virtually almost trivial.\end{enumerate}
Moreover, in this case $\aann_G(H)$ and $\aann_H(G)$ are definable.\ec
\bpf Conditions (1) and (2) are equivalent by Proposition \ref{symmetry}.

Suppose (1) and (2) hold. Put $A_n=\{g\in G:|H:\ann_H(g)|\le n\}$. Then each $A_n$ is definable, and
$\aann_G(H)=\bigcup_{n<\omega}A_n$. By compactness and (1), there are $n,k<\omega$ such that there are no 
$k$ disjoint translates of $A_n$ by elements in $G$. Let $A=\bigcup_i a_i+A_n$ be a maximal union of disjoint translates 
of $A_n$ by elements of $\aann_G(H)$. So for any $a\in\aann_G(H)$ we have $(a+A)\cap A\not=\emptyset$, 
whence $a\in A-A$. Thus $\aann_G(H)=A-A$ is definable; it follows that there is a finite bound on $|H:\ann_H(a)|$ for all $a\in\aann_G(H)$, as otherwise by sufficient saturation we could find $a\in\aann_G(H)$ such that $|H:\ann_H(a)|$ is infinite, a contradiction. 
By symmetry, the same holds for $\aann_H(G)$. Proposition \ref{almost_trivial} now implies that $\lambda$ restricted to 
$\aann_G(H)\times\aann_H(G)$ is almost trivial, so $\lambda$ is virtually almost trivial.

Conversely, if $\lambda$ is virtually almost trivial as witnessed by $G_0$ and $H_0$, then $\ann_{H_0}(g)$ has finite index in $H_0$ for $g\in G_0$, so $\ann_H(G)$ has finite index in $H$. Thus $G_0\le\aann_G(H)$, whence $G\ale\aann_G(H)$. Similarly $H\ale\aann_H(G)$.\epf

\section{Virtual almost triviality}\label{vat}
\bd Let $G$ be an infinite definable group. A complete type $p\in S_G(A)$ is {\em subgroup-generic} if $p$ is in no definable coset of a subgroup of infinite index in $G$ which has only finitely many images under automorphisms fixing $A$ (so it is $\acl^{eq}(A)$-definable). 
A sequence $\langle g_i:i\in I\rangle$ is {\em subgroup-generic} over $A$ if $\tp(g_i/A,\{g_j:j<i\})$ is subgroup-generic for all $i\in I$.\ed
Note that by Neumann's Lemma (\cite{BHN}) the group $G$ is not in the %filter 
ideal of definable sets covered by finitely many cosets of definable subgroups of $G$ of infinite index. By a standard construction (as, for example, in \cite[Fact 2.1.3]{K}), $G$ has a subgroup-generic complete type over any set of parameters. By Ramsey's Theorem and compactness,  indiscernible subgroup-generic sequences of any order type exist.

The following notion provides a useful replacement of the notion of principal generic type in a context where no connected components exist.
\bd Let $G$ be an infinite definable group and $A$ a set of parameters. An $A$-indiscernible sequence $\langle (g_i,\bar a_i):i\in I\rangle$ is {\em principal indiscernible} if  for any $i\in I$ and $A_i=A\cup\{g_j,\bar a_j:j\not=i\}$,
whenever $C$ is an $A_i$-definable coset of some subgroup $H$ and $g_i\in C$, then $g_i\in H^0_{A_i}$, the connected component of $H$ over $A_i$. It is {\em principal subgroup-generic} if moreover $\tp(g_i/A,\{g_j,\bar a_j:j<i\})$ is subgroup-generic for all $i\in I$.\ed
Note that $H=C^{-1}C$ is also definable over $A_i$.
\bp\label{connected} Principal subgroup-generic sequence exist. More precisely,
let $\epsilon>0$ be infinitesimal, let $\langle(x_i,\bar a_i):i\in\Q\cup(\Q+\epsilon)\rangle$ be an $A$-indiscernible sequence, and put $h_i=g_{i+\epsilon}^{-1}g_i$. 
Then $\langle(h_i,\bar a_i\bar a_{i+\epsilon}):i\in\Q\rangle$ is principal indiscernible over $A$; if moreover $\tp(g_i/A,\{g_j,\bar a_j:j<i\})$ is subgroup-generic for all $i$, then $\langle(h_i,\bar a_i\bar a_{i+\epsilon}):i\in\Q\rangle$ is principal subgroup-generic over $A$.\ep
It follows by compactness and indiscernibility that there are principal subgroup-generic sequences of any order-type.
\bpf Let $J\subset\Q\setminus\{i\}$ be a finite subset such that $C$ is definable over $X_J=A\cup\{g_j,g_{j+\epsilon},\bar a_j,\bar a_{j+\epsilon}:j\in J\}$, let $]m,M[$ be an open interval containing $\{i,i+\epsilon\}$ with $]m,M[\cap J=\emptyset$, and put $I=]m,M[\cap(\Q\cup(\Q+\epsilon))$. Let $H_0$ be an $X_J$-definable subgroup of finite index in $H$. 
Since $h_i=g_{i+\epsilon}^{-1}g_i\in C$, by indiscernibility of $\langle g_j:j\in I\rangle$ over $X$ we obtain $g_k^{-1}g_j\in C$ for all $j<k$ in $I$. By Ramsey's Theorem there is an infinite set of indices $I'\subseteq I$ such that all $g_j^{-1}g_k$ with $j<k$ in $I'$ are in the 
same coset $C_0$ modulo $H_0$. So for $j<k<\ell$ in $I'$ we obtain 
$$g_k^{-1}g_j=g_k^{-1}g_\ell g_\ell^{-1}g_j=(g_\ell^{-1}g_k)^{-1}g_\ell^{-1}g_j\in C_0^{-1}C_0=H_0.$$
By indiscernibility again $g_k^{-1}g_j\in H_0$ for all $j<k$ in $I$. In particular $h_i=g_{i+\epsilon}^{-1}g_i\in H_0$. As this is true for all finite $J\subset\Q\setminus\{i\}$ and all $X_J$-definable subgroups of $H$ of finite index, and since $A_i\subseteq\dcl(X_J:J\subset\Q\setminus\{i\}\mbox{ finite})$, we get $h_i\in H^0_{A_i}$.

Moreover, if $\tp(g_{i+\epsilon}/A,\{g_j,\bar a_j:j\le i\})$ is subgroup-generic, then $g_{i+\epsilon}$ is in no definable coset of a subgroup of infinite index in $G$ which has only finitely many images under automorphisms fixing $A\cup\{g_j,\bar a_j:j\le i\}$, and the same is true for $g_{i+\epsilon}^{-1}$, since right cosets of a subgroup $H$ of infinite index are left cosets of a conjugate of $H$, which still has infinite index. But then the same still is true of $g_{i+\epsilon}^{-1}g_i$, as we can just translate by $g_i^{-1}$ on the right. Thus  $\tp(h_i/A,\{g_j,\bar a_j:j\le i\})$ is subgroup-generic, and $\langle(h_i,\bar a_i\bar a_{i+\epsilon}):i\in\Q\rangle$ is subgroup-generic over $A$.\epf

\bt\label{trivial_by_finite}
Let $G$, $H$ and $K$ be abelian groups of finite burden definable in some $\omega$-categorical theory, and let $\lambda:G\times H\to K$ be  a definable bilinear quasi-form. Then $\lambda$ is virtually almost trivial. If $G$ and $H$ are connected, then $\lambda$ is trivial.\et
\bpf Let $\lambda:G\times H\to K$ be a counter-example. Define the {\em reduced burden} of $G$ and of $H$ with respect to $\lambda$ to be
$$\rbdn_\lambda(G)=\max_{\bar h\in H\text{ finite}}\bdn(G/\ann_G(\bar h))\quad\mbox{and}\quad\rbdn_\lambda(H)=\max_{\bar g\in G\text{ finite}}\bdn(H/\ann_H(\bar g)).$$
Then $\rbdn_\lambda(G)\le\bdn(G)$ and $\rbdn_\lambda(H)\le\bdn(H)$.\setcounter{claim}{0}

\begin{claim} If $A\le G$, $B\le H$ and $C\le D\le K$ are definable subgroups such that the map
$$\bar\lambda:A\times B\to D/C$$
(induced by composing the restriction of $\lambda$ to $(A\times B)\cap\lambda^{-1}[D]$ with the quotient map $D\to D/C$) is still a bilinear quasi-form, then $\rbdn_{\bar\lambda}(A)\le\rbdn_\lambda(G)$, $\rbdn_{\bar\lambda}(B)\le\rbdn_\lambda(H)$ and $\bdn(D/C)\le\bdn(H)$.\end{claim}
\bpfc This is immediate for $\bdn(D/C)$; moreover
$$\begin{aligned}\rbdn_\lambda(G)&=\max_{\bar h\in H\text{ finite}}\bdn(G/\ann_G(\bar h))\ge\max_{\bar h\in B\text{ finite}}\bdn(G/\ann_G(\bar h))\\
&\ge\max_{\bar h\in B\text{ finite}}\bdn(A/\ann_A(\bar h))=\rbdn_{\bar\lambda}(A)\,,\end{aligned}$$
similarly $\rbdn_\lambda(H)\ge\rbdn_{\bar\lambda}(B)$.\epf

Since our structure has finite burden, we may assume that 
$$\rbdn_\lambda(G)+\rbdn_\lambda(H)+\bdn(K)$$
is minimal possible among all possible counter-examples. Adding finitely many parameters to the language, we may assume that everything is $\emptyset$-definable.

Let $\epsilon>0$ be infinitesimal, let $\langle (x_i,x'_i):i\in\Q\cup(\Q+\epsilon)\rangle $ 
be an $\emptyset$-indiscernible subgroup-generic sequence in $G\times H$,
and put $y_i=x_i-x_{i+\epsilon}$ and $y'_i=x'_i-x'_{i+\epsilon}$. 
Then $\langle (y_i,y'_i):i<\omega\rangle $ is a principal subgroup-generic sequence over $\emptyset$ by Proposition \ref{connected} (in the first coordinate for $G$, in the second one for $H$).

\begin{claim}\label{images}  $\im\lambda_{y_i}$ and $\im\lambda_{y_j}$ are $\ale$-comparable for all $i<j$.\end{claim}
\bpfc Suppose not. By Lemma \ref{reducible} there is a minimal $2\le\ell\le\bdn(K)$ 
such that the sum $\sum_{i=0}^\ell\im\lambda_{y_i}$ is reducible (in the sense of Lemma \ref{reducible}). So there is $i_0\le\ell$ such that with $I=\{0,1,\ldots,\ell\}\setminus\{i_0\}$ we have
$$\im\lambda_{y_{i_0}}\ale\sum_{i\in I}\im\lambda_{y_i}=:C.$$
Consider the $\{y_j:j\in I\}$-definable subgroups
$$A=\{g\in G:\im\lambda_g\ale C\}\qquad\mbox{and}\qquad B=\{h\in H:\lambda'_h[A]\ale C\},$$
and note that $A=\aann_G^{\bar\lambda}(H)$ and $B=\aann_H^{\bar\lambda}(A)$, where $\bar\lambda:G\times H\to K/C$ is the induced bilinear quasi-form obtained by composing $\lambda$ with the quotient map. Then $H\ale B$ by Lemma \ref{symmetry}, so $B$ has finite index in $H$.

For every $i\in I$ there is a $\{y_j:j\in I\}$-definable induced bilinear quasi-form
$$\lambda_i:A\times B\to C/\im\lambda_{y_i}.$$
Note that for $a\in A$ the domain 
$\dom(\lambda_i)_a=\lambda_a^{-1}[C]\cap B$
has finite index in $B$ by Lemma \ref{add_basics}(3), and likewise for $(\lambda_i)_b'$ with $b\in B$,
so $\lambda_i$ is indeed a bilinear quasi-form.

By irreducibility of the sum $\sum_{j\in I}\im\lambda_{y_j}$, the quotient 
$\im\lambda_{y_i}/(\im\lambda_{y_i}\cap\sum_{j\in I,j\not=i}\im\lambda_{y_j})$ is infinite. Hence, as
$$(\im\lambda_{y_i}/(\im\lambda_{y_i}\cap\sum_{j\in I,j\not=i}\im\lambda_{y_j}))\times (\sum_{j\in I,j\not=i}\im\lambda_{y_j}/(\im\lambda_{y_i}\cap\sum_{j\in I,j\not=i}\im\lambda_{y_j}))$$
embeds definably into $C/(\im\lambda_{y_i}\cap\sum_{j\in I,j\not=i}\im\lambda_{y_j})$, Remark \ref{bdn_product} implies the 
following strict inequality:
$$\begin{aligned}\bdn(C/\im\lambda_i)&=\bdn(\sum_{j\in I,j\not=i}\im\lambda_{y_j}/(\im\lambda_{y_i}\cap\sum_{j\in I,j\not=i}\im\lambda_{y_j}))\\
&<\bdn(C/(\im\lambda_{y_i}\cap\sum_{j\in I,j\not=i}\im\lambda_{y_j}))\le\bdn(K).\end{aligned}$$
By induction, the bilinear quasi-form $\lambda_i$ is virtually almost trivial. Hence
the almost annihilator $\aann_A^{\lambda_i}(B)$ of $B$ with respect to the 
quasi-form $\lambda_i$ is an $\{y_j:j\in I\}$-definable subgroup of $A$ of finite index. 
Since $y_{i_0}\in A^0_{\{y_j:j\in I\}}$ by principal indiscernibility,
we obtain that $y_{i_0}\in \aann_A^{\lambda_i}(B)$. Thus
$\lambda_{y_{i_0}}[B]\ale\im\lambda_{y_i}$; as $B$ has finite index in $H$ we also have $\im\lambda_{y_{i_0}}\ale\im\lambda_{y_i}$.
The claim now follows from indiscernibility.\epf
\begin{claim}\label{kernels} $\ann_H(y_i)$ and  $\ann_H(y_j)$ are $\ale$-comparable for all $i<j$.\end{claim}
\bpfc Suppose not. By Lemma \ref{reducible-cap} there is a minimal $2\le\ell\le\bdn(H)$ such that the intersection 
$\bigcap_{i=0}^\ell\ann_H(y_i)$ is reducible. So there is $i_0\le\ell$ such that  $$B:= \bigcap_{i\in I}\ann_H(y_i)\ale\ann_H(y_{i_0}),$$ where
$I=\{0,1,\ldots,\ell\}\setminus\{i_0\}$.
Consider the $\{y_j:j\in I\}$-definable subgroup
$$A=\{g\in G:B\ale\ann_H(g)\}.$$
For every $i\in I$ consider the restricted bilinear quasi-form
$\lambda_i:A\times\ann_H(y_i)\to K$.
As $(\bigcap_{j\in I,j\not=i}\ann_H(y_j))/B$ is infinite by minimality of $\ell$,
we get by Remark \ref{bdn_product} and the definitions of $A$ and $B$ that
$$\begin{aligned}\rbdn_{\lambda_i}(\ann_H(y_i))&=\bdn(\ann_H(y_i)/B)<\bdn((\ann_H(y_i)+\bigcap_{j\in I,j\not=i}\ann_H(y_j))/B)\\
&\le\bdn(H/B)\le\rbdn_\lambda(H).\end{aligned}$$
By induction, the bilinear quasi-form $\lambda_i$ is virtually almost trivial. Hence 
$\aann_A^{\lambda_i}(\ann_H(y_i))$ is a subgroup of $A$ of finite index definable over $\{y_j:j\in I\}$.
Since $y_{i_0}\in A^0_{\{y_j:j\in I\}}$ by principal indiscernibility, we get that
$\ann_H(y_i)\ale\ann_H(y_{i_0})$.
The claim now follows from indiscernibility.\epf
Note that Claims \ref{images}.\ and \ref{kernels}.\ do not use subgroup-genericity of the sequence, only principal indiscernibility.
We will use this observation to apply (the proofs of) these claims below to certain forms induced by $\lambda$.

\begin{claim}\label{monotone} For $i<j$ we have $\ann_H(y_j)\ale\ann_H(y_i)$, and if $B\le H$ is definable over $\{y_k,y'_k:k\notin[i,j]\}$, then $\lambda_{y_i}[B]\ale\lambda_{y_j}[B]$. In particular $\im\lambda_{y_i}\ale\im\lambda_{y_j}$.\end{claim}
\bpfc By $\omega$-categoricity there is a bound $n$ on the index of $\ann_G(h)$ in $G$ for $h\in\aann_H(G)$. 
Choose $h\in\aann_H(G)$ subgroup-generic over $x_0,\ldots,x_n$. Then $x_i-x_j\in\ann_G(h)$ for some $0\le i<j\le n$, whence $h\in\ann_H(x_i-x_j)$. 
By subgroup-genericity of $h$ over $x_0,\ldots,x_n$, the group $\ann_H(x_i-x_j)\cap\aann_H(G)$ must have a finite index in $\aann_H(G)$. Thus $\aann_H(G)\ale\ann_H(x_i-x_j)$; by indiscernibility $\aann_H(G)\ale\ann_H(x_0-x_\epsilon)=\ann_H(y_0)$.

Suppose $\ann_H(y_0)\ale\ann_H(y_1)$. Then $y_1\in\aann_G(\ann_H(y_0))$; as $y_1$ is subgroup-generic over $y_0$ we have $G\ale\aann_G(\ann_H(y_0))$. By Proposition \ref{symmetry}, $\ann_H(y_0)\ale\aann_H(G)$. It follows that $\ann_H(y_i)\aeq\aann_H(G)$ for all $i\in\omega$, and $\ann_H(y_1)\aeq\ann_H(y_0)$. 

The first assertion now follows from Claim \ref{kernels}.

For the second assertion, let $\{y_k,y'_k:k\in I\}$ be the finitely many parameters needed to define $B$.
Put $m=\max I\cap(-\infty,i)$ and $M=\min I\cap(j,\infty)$. We can apply Claim \ref{images} to the restriction
of $\lambda$ to $G\times B$ and the sequence $\langle (y_k,y'_k):m<k<M\rangle$. 
Hence $\lambda_{y_i}[B]\ale\lambda_{y_j}[B]$ or $\lambda_{y_j}[B]\ale\lambda_{y_i}[B]$.
But $\ann_B(y_j)\ale\ann_B(y_i)$ by the first part, so if the second option holds,
then Lemma \ref{kerim} yields $\lambda_{y_0}[B]\aeq\lambda_{y_1}[B]$.\epf

\begin{claim}\label{preimages} If $i<j$ and $C\le K$ is definable over $\{y_k,y'_k:k\notin[i,j]\}$, then $\lambda^{-1}_{y_j}[C]\ale\lambda^{-1}_{y_i}[C]$.\end{claim}
\bpfc Consider the induced bilinear quasi-form $\bar\lambda:G\times H\to K/C$, and note that $\lambda^{-1}_y[C]=\ker\bar\lambda_y$. Then $\lambda^{-1}_{y_i}[C]$ and $\lambda^{-1}_{y_j}[C]$ are $\ale$-comparable by Claim \ref{kernels}; suppose $\lambda^{-1}_{y_i}[C]\ale\lambda^{-1}_{y_j}[C]$. Then $\lambda_{y_i}$ and $\lambda_{y_j}$ induce quasi-homomorphisms from $B:=\lambda^{-1}_{y_i}[C]$ to $C$. As $\ann_B(y_j)\ale\ann_B(y_i)$ by Claim \ref{monotone} and $\lambda_{y_j}[B]\ale\lambda_{y_i}[B]=C$, Proposition \ref{kerim} implies $\lambda_{y_j}[B]\aeq C$.
Thus 
$$\lambda^{-1}_{y_j}[C]\aeq B+\ann_H(y_j)\ale B+\ann_H(y_i)=B=\lambda^{-1}_{y_i}[C].\qedhere$$\epf

We shall now study $\lambda_{y_i,y_j}$ for $i<j$. By Claim~\ref{monotone} and Remark \ref{quasi-end} it induces a quasi-endomorphism $\bar\lambda_{y_i,y_j}$ of $H/\ann_H(y_j)$. 
By Corollary \ref{invertible}, any definable quasi-endomorphism of $H/\ann_H(y_j)$ with finite kernel must be almost surjective, and any definable almost surjective quasi-endomorphism must have finite kernel; these are precisely the invertible quasi-endomorphisms.
\begin{claim}\label{nilpotent} If $f$ is a definable quasi-endomorphism of $H/\ann_H(y_j)$, then $f$ is invertible or nilpotent.\end{claim}
\bpfc By Lemma \ref{sum}, we have an almost direct decomposition $H/\ann_H(y_j)\aeq \im f^{\circ n}+\ker f^{\circ n}$ for some $n<\omega$. 
Put $A=\{g\in G:\ann_H(y_j)\ale\ann_H(g)\}$, and let $B_1,B_2\le H$ be the preimages of $ \im f^{\circ n}$ and $\ker f^{\circ n}$, respectively. If $f$ were neither invertible nor nilpotent, then both summands are infinite.
For $i=1,2$ consider the restricted bilinear quasi-forms
$$\lambda_i:A\times B_i\to K.$$ By Remarks \ref{bdn_product} and \ref{bdn_cover} and the definition of $A$ we have
$$\begin{aligned}\rbdn_{\lambda_i}(B_i)&=\bdn(B_i/\ann_H(y_j))=\bdn(B_i/(B_1\cap B_2))\\&<\bdn((B_1+B_2)/(B_1\cap B_2))=\bdn(H/\ann_H(y_j))\le\rbdn_\lambda(H).\end{aligned}$$
By induction $\lambda_1$ and $\lambda_2$ are almost trivial, and so is the restriction
$$\lambda:A\times H\to K.$$ 
As $y_i\in A^0_{y_j}$ for $i<j$ by Claim \ref{monotone} and principal indiscernibility, the image 
$\im\lambda_{y_i}$ is finite and $y_i\in \aann_G(H)$. By subgroup-genericity of $y_i$ we get that
$\aann_G(H)\aeq G$, and $\lambda$ is virtually almost trivial by Corollary \ref{aann}, a contradiction.
\epf

\begin{claim} For $i<j<k$ we have $\lambda_{y_i,y_k}\equiv\lambda_{y_j,y_k}\circ\lambda_{y_i,y_j}$.\end{claim}
\bpfc Note that $\lambda_g\circ\lambda_g^{-1}\equiv\id_{\im\lambda_g}$ for any $g\in G$ by Remark \ref{inverse}. Hence
$$\begin{aligned}\lambda_{y_0,y_j}&=\lambda_{y_j}^{-1}\circ\lambda_{y_0}
=\lambda_{y_j}^{-1}\circ\id_{\im\lambda_{y_0}}\circ\lambda_{y_0}\\
&\equiv\lambda_{y_j}^{-1}\circ\id_{\im\lambda_{y_i}}\circ\lambda_{y_0}\equiv
\lambda_{y_j}^{-1}\circ\lambda_{y_i}\circ\lambda_{y_i}^{-1}\circ\lambda_{y_0}=\lambda_{y_i,y_j}\circ\lambda_{y_0,y_i}.\qedhere\end{aligned}$$\epf

\begin{claim}\label{noteq} For $i\not=j$ we have $\ann_H(y_i)\not\aeq\ann_H(y_j)$.\end{claim}
\bpfc Suppose otherwise. Then $\ann_H(y_i)\aeq\ann_H(y_j)$ for all $i,j\in\mathbb Q$.
Let $R$ be the ring of definable quasi-endomorphisms of $\bar H=H/\ann_H(y_0)$. Note that $\bar H$ is infinite, as $\ann_H(y_0)$ has infinite index in 
$H$ by assumption (otherwise, as above, we get that
$\lambda$ is virtually almost trivial by Corollary \ref{aann}).

It follows from Claim \ref{nilpotent} that the set of nilpotent quasi-endomorphisms of $\bar H$ is an ideal: it is clearly invariant under left and right multiplication; if $f$ and $g$ are nilpotent but $f+g$ is not nilpotent, there is invertible $h$ with $h(f+g)=hf+hg=\id$. So $hf=\id-hg$ is nilpotent. 
But $(\id-hg)(\id+hg+(hg)^2+(hg)^3+\cdots)=\id$ (note that the sum is finite, as $hg$ is nilpotent), so $hf=\id-hg$ is invertible, a contradiction.
Thus $R/I$ is a division ring, which is locally finite by $\omega$-categoricity, whence a locally finite field by Wedderburn's Theorem.

Consider $0<i<j$. As $\ann_H(y_0)\aeq\ann_H(y_i)\aeq\ann_H(y_j)$, the quasi-endomorphism $\bar\lambda_{y_i,y_j}$ has finite kernel, and must be invertible. By local finiteness and indiscernibility, it has a fixed finite multiplicative order $N$ modulo $I$.
Hence there are only finitely many possibilities for $\hat\lambda_{y_i,y_j}+I$ (where $\hat\lambda$ is the equivalence class of $\bar\lambda$ in $R$).
By indiscernibility, $\hat\lambda_{y_i,y_j}+I$ does not depend on $i,j$. But $\hat\lambda_{y_j,y_k}\cdot\hat\lambda_{y_i,y_j}=\hat\lambda_{y_i,y_k}$ for
$i<j<k$, whence 
$$\hat\lambda_{y_i,y_j}+I=(\hat\lambda_{y_i,y_{i+{(j-i)/N}}})^N+I\in\id_{\bar H}+I.$$
By indiscernibility, $\hat\lambda_{x_i-x_j,x_k-x_\ell}\in\id_{\bar H}+I$ for all $0<i<j<k<\ell$ in $\omega$.
Now 
$$\hat\lambda_{x_1-x_3,x_2-x_3}=\hat\lambda_{x_2-x_3,x_4-x_5}^{-1}\cdot\hat\lambda_{x_1-x_3,x_4-x_5}\in\id_{\bar H}+I.$$
Let $B=\im(\lambda_{x_1-x_3,x_2-x_3}-\id_H)$, a definable subgroup of infinite index in $H$ almost containing $\ann_H(y_0)$.
Then for all $h\in H^0_{x_1,x_2,x_3}$ there is $b\in B$ with $h+b\in\lambda_{x_1-x_3,x_2-x_3}(h)$
(as $\dom\lambda_{x_1-x_3,x_2-x_3}$ is a $\{x_1,x_2,x_3\}$-definable subgroup of $H$ of finite index). Hence $\lambda_{x_2-x_3}(h+b)=\lambda_{x_1-x_3}(h)$, that is
$$\lambda(x_1-x_3,h)=\lambda(x_2-x_3,h+b)=\lambda(x_2-x_3,h)+\lambda(x_2-x_3,b),$$
whence
$$\lambda(x_1-x_2,h)=\lambda((x_1-x_3)-(x_2-x_3),h)=\lambda(x_2-x_3,b).$$
But this means that $\im\lambda_{x_1-x_2}\ale\lambda_{x_2-x_3}[B]$. On the other hand, as $\lambda_{y_i,y_j}$ is a
quasi-isomorphism of $\bar H$ for $i<j$, so is $\lambda_{x_i-x_j,x_k-x_\ell}$ for all $i<j<k<\ell$. In particular
$$\im\lambda_{x_2-x_3} \aeq\im\lambda_{x_4-x_5}\aeq\im\lambda_{x_1-x_2}\ale \lambda_{x_2-x_3}[B].$$
But $B$ has infinite index in $H$ and $\ker\lambda_{x_2-x_3}\ale B$,
so $B/(B\cap\ker\lambda_{x_2-x_3})$ has infinite index in $H/(B\cap\ker\lambda_{x_2-x_3})$,
and both of them are definably quasi-isomorphic to $\im\lambda_{x_2-x_3}$, hence definably quasi-isomorphic to each other. This contradicts Corollary \ref{invertible}(2).
\epf

Note that for $i<j<k<\ell$ and $B\le H$ definable over $\{y_s,y'_s:s\notin[j,k]\}$ we have $\lambda_{y_j,y_\ell}[B]\ale\lambda_{y_k,y_\ell}[B]$ by Claim \ref{monotone}, and $\lambda_{y_i,y_k}[B]\ale\lambda_{y_i,y_j}[B]$ by Claim \ref{preimages}.

\begin{claim}\label{L-images} If $i<j<k<\ell$ and $B\le H$ is definable over $\{y_s,y'_s:s\not=i,j,k,\ell\}$ then $\lambda_{y_i,y_j}[B]$ and $\lambda_{y_k,y_\ell}[B]$ are $\ale$-comparable.\end{claim}
\bpfc Suppose not, and put 
$$A=\{g\in G:\lambda_g[B]\ale\lambda_{y_k}[B]\}\qquad\mbox{and}\qquad B'=\{h\in B:\lambda'_h[A]\ale\lambda_{y_k}[B]\}.$$
Then $A=\aann_G^{\bar\lambda}(B)$ and $B'=\aann_B^{\bar\lambda}(A)$, where  $\bar\lambda:G\times B\to K/\lambda_{y_k}[B]$ is the induced bilinear quasi-form. Then $B\ale B'$ by Lemma \ref{symmetry}, so $B'$ has finite index in $B$.

Consider the induced bilinear quasi-form
$$\tilde\lambda:A\times B'\to\lambda_{y_k}[B]/( \lambda_{y_k}[B]\cap\lambda_{y_\ell}\lambda_{y_i,y_j}[B]).$$
As $\lambda_k[B]\ale\lambda_\ell[B]$, we have
$$\begin{aligned}\bdn(\lambda_{y_k}[B]/( \lambda_{y_k}[B]\cap\lambda_{y_\ell}\lambda_{y_i,y_j}[B]))&
=\bdn(\lambda_{y_\ell}\lambda_{y_k,y_\ell}[B]/(\lambda_{y_\ell}\lambda_{y_k,y_\ell}[B]\cap
\lambda_{y_\ell}\lambda_{y_i,y_j}[B])\\
&\le\bdn(\lambda_{y_k,y_\ell}[B]/(\lambda_{y_i,y_j}[B]\cap
\lambda_{y_k,y_\ell}[B]))\\
&<\bdn((\lambda_{y_i,y_j}[B]+\lambda_{y_k,y_\ell}[B])/(\lambda_{y_i,y_j}[B]\cap
\lambda_{y_k,y_\ell}[B]))\\
&\le\bdn(H/\ann_H(y_\ell))=\bdn(\im\lambda_{y_\ell})\le\bdn(K),\end{aligned}$$
the bilinear quasi-form $\tilde\lambda$ is virtually almost trivial by induction. Since $y_{k'}\in A^0_{y_i,y_j,y_k,y_\ell}$ for some 
$j<k'<k$ such that $B$ is definable over $\{y_s,y'_s:s\notin[k',k]\cup\{i,j,\ell\}\}$ by Claim \ref{monotone} and principal indiscernibility, it follows that
$$\lambda_{y_{k'}}[B]\ale\lambda_{y_\ell}\lambda_{y_i,y_j}[B].$$
Hence $\lambda_{y_{k'},y_\ell}[B]\ale\lambda_{y_i,y_j}[B]+\ann_B(y_\ell)\ale\lambda_{y_i,y_j}[B]$, and $\lambda_{y_k,y_\ell}[B]\ale\lambda_{y_i,y_j}[B]$ by indiscernibility.\epf

\begin{claim}\label{L-kernels} $\ker\bar\lambda_{y_i,y_j}$ and $\ker\bar\lambda_{y_k,y_\ell}$ are $\ale$-comparable for all $i<j<k<\ell$, where $\bar\lambda_{y,y'}$ is the quasi-homomorphisms from $H$ to $H/\ann_H(y_j)$ induced by $\lambda_{y,y'}$.\end{claim}
\bpfc We have $\ker\bar\lambda_{y_i,y_j}=\ann_H(y_i)$; put $B=\ker\bar\lambda_{y_k,y_\ell}$ and suppose that they are not $\ale$-comparable. Let $A=\{g\in G:\ann_H(y_i)\ale\ann_H(g)\}$ and consider the restricted bilinear quasi-form
$$\bar\lambda:A\times B\to K.$$
Since $\ann_H(y_k)\le B\cap\ann_H(y_i)$, we have
$$\begin{aligned}\rbdn_{\bar\lambda}(B)&=\bdn(B/(B\cap\ann_H(y_i)))<\bdn((B+\ann_H(y_i))/(B\cap\ann_H(y_i)))\\
&\le\bdn(H/\ann_H(y_k))\le\rbdn_\lambda(H).\end{aligned}$$
Hence $\bar\lambda$ must be virtually almost trivial by induction. Since $y_s\in A^0_{y_i,y_j,y_k,y_\ell}$ for $s<i$, we have $\ker\lambda_{y_k,y_\ell}=B\ale\ann_H(y_s)$ for all $s<i$; the claim now follows from indiscernibility.\epf

\begin{claim}\label{case1} If $\im\lambda_{y_0,y_1}\ale\im\lambda_{y_2,y_3}$ then 
$\im\lambda_{y_0,y_j}\ale\im\lambda^{\circ n}_{y_i,y_j}$ for all $0<i<j$ and $1\le n<\omega$.\end{claim}
\bpfc We proceed by induction on $n$. For $n=1$ this is clear, as 
$\im\lambda_{y_0,y_j}\ale\im\lambda_{y_i,y_j}$ by Claim \ref{monotone}. Assume it holds for some $n$. 
Choose $0<k<\ell<i$. Then
$$\begin{aligned}\im\lambda_{y_0,y_j}&
\aeq\im(\lambda_{y_i,y_j}\circ\lambda_{y_0,y_i})
=\lambda_{y_i,y_j}[\im\lambda_{y_0,y_i}]\ale\lambda_{y_i,y_j}[\im\lambda_{y_0,y_k}]\\
&\ale\lambda_{y_i,y_j}[\im\lambda_{y_\ell,y_j}]\ale\lambda_{y_i,y_j}[\im\lambda_{y_i,y_j}^{\circ n}]
=\im(\lambda_{y_i,y_j}\circ\lambda_{y_i,y_j}^{\circ n})=\im\lambda_{y_i,y_j}^{\circ (n+1)}\end{aligned}$$
(the first inequality follows by the second part of Claim \ref{monotone}, the second inequality follows by the assumption of
the claim, and the last one by the inductive assumption).
\epf

\begin{claim}\label{case2} If $\im\lambda_{y_2,y_3}\ale\im\lambda_{y_0,y_1}$ then $\im\lambda_{y_0,y_k}\ale\im\lambda^{\circ n}_{y_i,y_j}$ for all $0<i<j<k$ and $1\le n<\omega$.\end{claim}
\bpfc The case $n=1$ follows from Claims \ref{monotone} and \ref{preimages}, so assume the statement holds
for some $n$. Choose $0<i<j<\ell<m<k$. Let $\bar\lambda_{y,y'}$ be the quasi-homomorphism from 
$H$ to $H/\ann_H(y_j)$ induced by $\lambda_{y,y'}$. 
By the assumption and indiscernibility we have $\im\bar{\lambda}_{y_m,y_k}\ale\im\bar{\lambda}_{y_i,y_j}$.
Hence, Claim 
\ref{L-kernels}  and Lemma \ref{kerim} imply $\ker\bar\lambda_{y_i,y_j}\ale\ker\bar\lambda_{y_m,y_k}$,
so the same holds for
the restrictions to $B:=\im\lambda_{y_0,y_\ell}$. But now by Lemma \ref{kerim} and Claim \ref{L-images} we have $\lambda_{y_m,y_k}[B]\ale\lambda_{y_i,y_j}[B]$. Then
$$\begin{aligned}\im\lambda_{y_0,y_k}&\aeq\im(\lambda_{y_m,y_k}\circ\lambda_{y_0,y_m})=
\lambda_{y_m,y_k}[\im\lambda_{y_0,y_\ell}]\ale\lambda_{y_i,y_j}[\im\lambda_{y_0,y_\ell}]\\
&\ale\lambda_{y_i,y_j}[\im\lambda_{y_i,y_j}^{\circ n}]=\im(\lambda_{y_i,y_j}^{\circ n}\circ\lambda_{y_i,y_j})=\im\lambda_{y_i,y_j}^{\circ (n+1)}.\qedhere\end{aligned}$$\epf

\begin{claim}\label{ker} $\im\lambda_{y_i,y_k}\ale\ann_H(y_j)$ for all $i,j<k$.
\end{claim}
\bpfc By Claim \ref{noteq}, the quasi-endomorphism $\bar\lambda_{y_i,y_j}$ of $H/\ann_H(y_j)$ induced by $\lambda_{y_i,y_j}$ is not invertible, so it must
be nilpotent by Claim \ref{nilpotent}. The assertion now follows from Claims \ref{L-images}, \ref{case1} and \ref{case2}.\epf

Of course, all of the previous claims also hold with the roles of $G$ and $H$ exchanged.

\begin{claim}\label{comparable} For any $i\not=j$ we have $\im\lambda_{y_i}\ale\im\lambda'_{y'_j}$ or  $\im\lambda'_{y_j}\ale\im\lambda_{y_i}$.
\end{claim}
\bpfc Suppose not. Put $A=\{g\in G:\im\lambda_g\ale\im\lambda_{y_i}\}$, and consider the induced bilinear quasi-form
$$\bar\lambda:A\times H\to\im\lambda_{y_i}/(\im\lambda_{y_i}\cap\im\lambda'_{y'_j}).$$
As $\bdn(\im\lambda_{y_i}/(\im\lambda_{y_i}\cap\im\lambda'_{y'_j}))<\bdn(K)$, the quasi-form $\bar\lambda$ must be virtually almost trivial. But $y_k\in A^0_{y_i,y'_j}$ for $j\not=k<i$ by Claim \ref{monotone} and principal indiscernibility. Hence $\im\lambda_{y_k}\ale\im\lambda'_{y'_j}$, a contradiction, as there is $j\not=k<i$ with $y_k\equiv_{y'_j}y_i$.\epf

By Claim \ref{comparable} and symmetry we may assume that $\im\lambda'_{y'_i}\ale\im\lambda_{y_j}$ for all $i<j$.
Fix $i$ and $k$, and choose $i<j<\ell$ and $k<\ell<1$ with $k\notin\{i,j\}$. By Claim \ref{ker} we get:
$$\im\lambda_{y_j}\ale\lambda_{y_\ell}[\ann_H(y_k)].$$
Moreover, $\lambda_{y_\ell}[\ann_H(y_k)] \ale\lambda_{y_1}[\ann_H(y_k)]$  by Claim \ref{monotone}.
Then
$$\lambda(y_1,y'_i)\in(\im\lambda'_{y'_i})^0_{y'_i,y_j,y_k,y_\ell}\le(\im
\lambda_{y_j})^0_{y'_i,y_j,y_k,y_\ell}\le \lambda_{y_\ell}[\ann_H(y_k)].$$
Hence,
$$y'_i\in \lambda_{y_1}^{-1}[\lambda_{y_\ell}[\ann_H(y_k)]^0_{y_k,y_\ell,y_1}]\leq 
\lambda_{y_1}^{-1}[\lambda_{y_1}[\ann_H(y_k)]].$$ 
Thus,
$$y'_i\in(\ann_H(y_k)+\ann_H(y_1))^0_{y_k,y_1}\le\ann_H(y_k),$$
and $\lambda(y_k,y'_i)=0$. As $\langle (y_i,y'_i):i<\omega \rangle$ is a subgroup-generic sequence, $\ann_G(y'_0)$ has finite index in $G$ and $\ann_H(y_0)$ has finite index in $H$. Since $(y_0,y'_0)$ is subgroup-generic, $\aann_H(G)$ and $\aann_G(H)$ have finite index in $H$ and $G$ respectively. So $\lambda$ is virtually almost trivial by Corollary \ref{aann}.

Finally, if $G$ and $H$ are connected, then $\lambda$ is almost trivial. But then for every $g\in G$ and $h\in H$ the annihilators $\ann_H(g)$ and $\ann_G(h)$ have finite index in $H$ and in $G$, 
and must be equal to $H$ and $G$, respectively, by connectivity. Thus $\lambda$ is trivial.
\epf

\section{On groups and rings}\label{main_results}
Recall that each countable, $\omega$-categorical group has a finite series of characteristic (i.e.\ invariant under
the automorphism group) subgroups in which all successive quotients are characteristically simple groups 
(i.e.\ they do not have non-trivial, proper characteristic subgroups). On the other hand, Wilson \cite{W} proved (see also \cite{A} for an exposition of the proof):

\begin{fact}\label{Wilson}
For each infinite, countable, $\omega$-categorical, characteristically simple group $H$, one of the following holds.\begin{enumerate}
\item[(i)] For some prime number $p$, $H$ is an elementary abelian $p$-group
(i.e.\ an abelian group, in which every nontrivial element has order $p$). 
\item[(ii)] $H \cong B(F)$ or $H \cong B^-(F)$ for some non-abelian, finite, simple group $F$,
where $B(F)$ is the group of all continuous functions from the Cantor space $C$ to $F$, and 
$B^-(F)$ is the subgroup of $B(F)$ consisting of the functions $f$ such that $f(x_0)=e$ for a fixed 
element $x_0 \in C$.
\item[(iii)] $H$ is a perfect $p$-group (perfect means that $H$ equals its commutator subgroup).
\end{enumerate}
\end{fact}

It remains a difficult open question whether there exist infinite, 
$\omega$-categorical, perfect $p$-groups.

\begin{remark}\label{B(F)}
 The groups $B(F)$ and $B^-(F)$ above have TP$_2$ (in particular, they do not have finite burden).
\end{remark}
\begin{proof}
 Let $f\in F$ be a non-central element, and let $\langle D_i:i<\omega\rangle$ be pairwise 
 disjoint clopen sets in $C$. Let $g_i\in B(F)$ be given by $g_i[A_i]=\{f\}$ and $g_i[C\backslash A_i]=\{0\}$ for each $i$.
 Then the centralizers of the $g_i$ do not satisfy the conclusion of \cite[Theorem 2.4]{CKS}, hence $B(F)$ has TP$_2$.
 The argument for $B^-(F)$ is analogous.
\end{proof}

\begin{fact}[{\cite[Theorem 3.1]{N}}]\label{derived}
There is a finite bound of the size of conjugacy classes in a group
$G$ if and only if the derived subgroup $G'$ is finite.
\end{fact}
This implies in particular that if the almost centre $\Z(G)$ of a group $G$ is definable, then it is finite-by-abelian.

\br\label{definability}
 If a group $G$ is virtually finite-by-abelian, then there is a characteristic definable finite-by-abelian subgroup 
 $G_0\leq G$ of finite index; if a ring $R$ is virtually finite-by-null, there is a definable subring $R_0$ which is finite-by-null.
\er
\bpf Let $G$ be virtually finite-by-abelian. Then $\Z(G)$ is characteristic and definable of finite index (this does not even need $\omega$-categoricity), and finite-by-abelian.
 
If $R$ is virtually finite-by-null, let $S_0$ be a finite-by-null subring of finite (additive) index, and $I$ a finite ideal of $S_0$ containing $S_0\cdot S_0$. Then $S:=\bigcap_{s\in S_0}\{r\in R:rs\in I\}$ contains $S_0$ and must be a definable subgroup of finite index, with $S\cdot S_0\subseteq I$. Now $R_0:=S\cap\bigcap_{s\in S}\{r\in R:sr\in I\}$ contains $S_0$ and is again a definable subgroup of finite index. Since $R_0\cdot R_0\subseteq I\le R_0$, this is
a required subring.
\epf

We will use the following variant of Proposition 2.5 from \cite{KLS}.
As in our context we cannot use connected components, we have to modify the proof slightly.
\bl\label{criterion}
Let $\mathcal{C}$ be a class of countable, $\omega$-categorical NTP$_2$ (pure) groups, closed under taking definable subgroups
and quotients by definable normal subgroups. Suppose that
every infinite, characteristically simple group in $\mathcal{C}$ is solvable. Then every group in $\mathcal{C}$ is nilpotent-by finite.
\el
\bpf
 Let $G\in \mathcal{C}$. Let $\{0\}=G_0\leq G_1\leq\dots\leq G_n=G$ be a chain of characteristic subgroups of $G$ of maximal 
 length. We will show the assertion by induction $n$. Let $i$ be maximal such hat $G_i$ is finite. Then $C_G(G_i)$ is a characteristic subgroup of $G$ of finite index, so we can replace $G$ by $C_G(G_i)/G_i$ without increasing $n$. We can thus assume that $G_1$ is infinite.
 Now, as $G_1$ is characteristically simple, it is solvable by the assumption. By the inductive hypothesis, $G/G_1$ is virtually nilpotent, 
 so there is a normal definable subgroup $N$ of $G$ of finite index such that $N/G_1$ is nilpotent, so $N$ is solvable. Since $N$ is NTP$_2$, it does not interpret the atomless boolean algebra, so by \cite[Theorem 1.2]{AM} it is virtually nilpotent, and so is $G$.
\epf

\bp\label{soluble}
A nilpotent $\omega$-categorical group of finite burden is virtually finite-by-abelian.
\ep
\bpf 
Let $G$ be a counter-example; we may assume it is nilpotent of minimal class possible. 
Then $Z(G)$ is infinite, and $G/Z(G)$ is virtually finite-by-abelian. By Remark \ref{definability} there is a 
definable subgroup $G_0$ of finite index and a finite normal subgroup $F/Z(G)$ of $G_0/Z(G)$ such that $G_0/F$ is abelian. Clearly we may replace $G$ by $C_{G_0}(F/Z(G))$, a definable subgroup of finite index. Then $G'\le F$ so 
$G'/Z(G)$ is central in $G/Z(G)$ ($*$), and $F\ale Z(G)$. Thus $F'$ is finite by Fact \ref{derived}; we may assume it is trivial.
Replacing $G$ by a definable subgroup of finite index, we may assume that the index $|G':G'\cap Z(G)|$ is not greater than 
$|G_0':G_0'\cap Z(G)|$ for any definable $G_0\leq G$ of finite index $(\dag)$.

Consider $g\in G$. By $(*)$, the map $x\mapsto [g,x]Z(G)$ is a definable homomorphism from $G$ to $G'/Z(G)$; its kernel $H$ must have finite index. Then $x\mapsto[g,x]$ is a definable homomorphism from $H$ to $Z(G)$ with abelian image;
its kernel must hence contain $H'Z(G)$. As $H'Z(G)=G'Z(G)$ by $(\dag)$, we see that $G'\le C_G(g)$. This holds for all $g\in G$, so $G'\le Z(G)$.

Now commutation is a definable bilinear form from $G/Z(G)$ to $Z(G)$. By Theorem \ref{trivial_by_finite} it is virtually almost trivial. 
But this means that $G$ is virtually finite-by-abelian, contradicting our assumption.\epf

\bp\label{group-nilpotent}
An $\omega$-categorical group of finite burden is nilpotent-by finite.
\ep
\bpf\setcounter{claim}{0}
If the proposition does not hold, there is a non-soluble $\omega$-categorical characteristically simple group $G$ by 
Lemma \ref{criterion}; it must be a perfect $p$-group for some prime $p$ by Fact \ref{Wilson} 
and Remark \ref{B(F)}.
We choose such a $G$ of minimal possible burden $k$. Note that $\Z(G)$ is trivial, as it is characteristic and 
finite-by-abelian (so soluble, as it is a $p$-group). Hence there are no finite normal subgroups, and all non-trivial conjugacy classes are infinite. %Then $k>1$ by \cite[Proposition 5.5]{DW}. 
%As burden does not increase when taking subgroups or quotients, every definable section of $G$ of strictly smaller burden 
%is virtually finite-by-nilpotent.
\begin{claim} The soluble radical $R(G)$ of $G$ is trivial.\end{claim}
\bpfc Suppose $R(G)$ is non-trivial. Then there is some non-trivial $a\in R(G)$ such that $a^G$ generates an infinite (definable) abelian normal subgroup $A_a$ of $G$. Let $\A=\{A_a:A_a \mbox{ abelian}\}$, a definable invariant  collection of definable abelian normal subgroups.

Any finite product $S$ of groups in $\A$ is nilpotent, whence virtually finite-by-abelian by Proposition \ref{soluble}. But then  $\Z(S)'$ is a finite characteristic subgroup of $S$, whence normal in $G$, and thus trivial. So $S$ is virtually abelian. It follows that for $A\in\A$ the almost centraliser $\C_G(A)$ almost contains $A'$ for all $A'\in\A$. Hence
$$A\ale\C_A(\C_G(A))\le\C_A(A')$$
(the first inequality follows from Fact \ref{Csymmetry}). But $[\C_A(A'),\C_{A'}(A)]$ is normal; aplying Proposition \ref{almost_trivial} to the bilinear form $(x,y)\mapsto[x,y]$ from $\C_A(A')\times \C_{A'}(A)$ to $A\cap A'$, we see that it is finite, whence trivial. Hence $\A'=\{\C_A(\C_G(A)):A\in\A\}$ is an invariant family of pairwise commuting abelian groups, and generates a characteristic abelian subgroup, which must be the whole of $G$. 
This contradiction finishes the proof of the claim.
\epf

Suppose every centraliser of a non-trivial element is soluble. Then by compactness there is a bound on the derived length of any proper centraliser.
As every finite subset of $G$ is contained in a centralizer of a nontrivial element, this would imply that $G$ is soluble, a contradiction. 
Hence there is $n\in G\setminus\{1\}$ such that $H:=C_G(n)$ is non-soluble. Put $N:=\langle n^G\rangle$, an infinite
normal subgroup, which is definable by $\omega$-categoricity.

Since $\C_G(N)\cap N$  is normal and finite-by-abelian (by Fact \ref{derived} and $\omega$-categoricity), 
whence soluble, it must be trivial.
\begin{claim}
 $\tilde{C}_G(N)=\{1\}$ .
\end{claim}
\begin{proof}
 Suppose $\tilde{C}_G(N)$ is nontrivial, whence infinite. The map $\tilde{C}_G(N)\times N  \to G$ 
 given by multiplication is a definable injection, so $\bdn(G)\ge\bdn(N)+\bdn(\C_G(N))$ by Remark \ref{bdn_product}. As $\tilde{C}_G(N)$
 is infinite, $\bdn(\C_G(N))\ge 1$, and $\bdn(N)<\bdn(G)=k$. By inductive hypothesis
 $N$ is nilpotent-by-finite, whence solvable, a contradiction.
\end{proof}

\begin{claim}
 Any definable normal subgroup $M$ of $G$ with $M \ale H$ is trivial.
\end{claim}
\begin{proof}
If $M\ale H=C_G(n)$, then $n\in\C_G(M)$; by normality of $M$ we get $n^G\subseteq
\C_G(M)$, and hence $N\leq\C_G(M)$. Then $M\ale\C_G(N)=\{1\}$ by Fact \ref{Csymmetry}.
\end{proof}
Consider a nontrivial definable normal subgroup $M$ of $G$. 
Since $M\cap H$ is normalized by $H$, we have a definable injection 
$$M/(M\cap H)\times H/(M\cap H)\to G/(M\cap H)$$
given by multiplication. As $M/M\cap H$ is infinite by the claim, we conclude that 
$$\bdn(H/M\cap H)<\bdn(G/(M\cap H))\le\bdn(G)=k.$$
By inductive hypothesis, $H/(M\cap H)$ is nilpotent-by-finite, whence soluble. If $M$ runs through the family $\mathcal{M}$ of 
$1$-generated normal subgroups, the family $\{H/(M\cap H):M\in\mathcal M\}$ is uniformly definable by $\omega$-categoricity, 
and by compactness there is $d<\omega$ such that $H/(M\cap H)$ has derived length at most $d$ for all $M\in\mathcal M$. 
But this means that $H^{(d)}$ is contained in $M$ for all $M\in\mathcal M$, and thus is contained in all nontrivial normal subgroups.

Since $H$ is not soluble, $H^{(d)}$ generates a non-abelian a minimal normal subgroup $L$. But then $L$ is finite by \cite[Theorem D]{A}, a contradiction.
This completes the proof.
\epf

\bt\label{groups}
An $\omega$-categorical group of finite burden is virtually finite-by-abelian.
\et
\bpf  This follows immediately from Propositions \ref{soluble} and \ref{group-nilpotent}.
\epf

\begin{corollary}\label{dp_min_gps}
An $\omega$-categorical NIP group of finite burden is virtually abelian. 
\end{corollary}
\begin{proof}
Let $G$ be an $\omega$-categorical NIP group of finite burden. 
By a result of Shelah, the absolute connected component 
$G^{00}$ (i.e.\ the smallest type-definable subgroup of
$G$ of bounded index) exists (see \cite[Theorem 6.1]{HPP} for a proof). By $\omega$-categoricity, $G^{00}$ is definable and hence of
finite index in $G$, so we may assume that
$G$ is connected. Then $G$ is finite-by-abelian by Remark \ref{definability}. Thus, the centralizer of any element
in $G$ has finite index in $G$, hence, by connectedness, is equal to $G$. 
This means that $G$ is abelian.
\end{proof}

\bt\label{rings}
An $\omega$-categorical ring of finite burden is virtually finite-by-null. 
\et
\bpf This is immediate from Theorem \ref{trivial_by_finite}, as multiplication is a definable bilinear map.
\epf
As for groups, we get a corollary:
\begin{corollary}\label{dp_min_rings}
An $\omega$-categorical NIP ring of finite burden is virtually null. 
\end{corollary}
\begin{proof}
 Let $R$ be such a ring. We may again assume that $R$ is connected (in the sense of the additive group).
 Then $R$ is finite-by-null by Remark \ref{definability}. Hence, the left annihilator of any element in $R$ has finite
 index in $R$, and must be equal to $R$ by connectedness. This shows that $R$ is null.
\end{proof}

\section{Questions and concluding remarks}\label{q+r}
One can ask various questions about generalizations of the above results
to more general contexts, such as strong or NTP$_2$ theories. For example, one can ask:
\begin{question}\label{q}
 Are $\omega$-categorical strong groups \\
 (1) virtually nilpotent-by-finite?\\
 (2) virtually abelian-by-finite?
\end{question}
An analogue of Question \ref{q}(1) for rings has positive answer by Theorem \ref{NTP2rings} below.
As to the stronger version, we do not know:
\begin{question}
  Are $\omega$-categorical strong rings null-by-finite?
\end{question}

The proof below is a modification of the proof of Theorem 2.1 from \cite{Kr},
generalizing that result from the NIP to the NTP$_2$ context.
\begin{theorem}\label{NTP2rings}
Every $\omega$-categorical NTP$_2$ ring is nilpotent-by-finite. 
\end{theorem}
\begin{proof}\setcounter{claim}{0}
As in \cite{Kr}, it is enough to show that a semisimple $\omega$-categorical
NTP$_2$ ring $R$ is finite, and we can assume that $R$ is a subring of $\prod_{i\in I}R_{i}$, where each $R_{i}$ is finite, and 
$|\{R_{i}:i\in I\}|<\omega$. Let $\pi_{i}$ be the projection onto the $i$-th coordinate.
For $i_{0},\dots,i_{n}\in I$ and $r_{0}\in R_{i_{0}},\dots,r_{n}\in R_{i_{n}}$, we define
$$R_{i_{0},\dots,i_{n}}^{r_{0},\dots,r_{n}}=\left\{r\in R : \bigwedge_{j=0}^{n}\pi_{i_{j}}(r)=r_{j}\right\}.$$
Suppose for a contradiction that $R$ is infinite.  Again as in \cite{Kr}, we get the following claim:
\begin{claim}
 For any $N\in \omega$ there are pairwise distinct $i(0),\dots,i(N-1)\in I$
 and non-nilpotent elements $r_i\in R_i$ for $i<N$ such that the sets
 $$R_{i_{0},\dots,i_{N-1}}^{r_{0},0\dots,0},R_{i_{0},\dots,i_{n}}^{0,r_1\dots,0},\dots,
 R_{i_{0},\dots,i_{n}}^{0,0\dots,r_{N-1}}$$ are all non-empty.
 \end{claim}

 Notice that, by $\omega$-categoricity, the principal two-sided ideals $RxR$ for $x\in R$ 
 are uniformly definable. Hence, by \cite[Theorem 2.4]{CKS} and compactness, we obtain 
 in particular that in order to contradict NTP$_2$ it is enough to 
 find for any $n,m<\omega$ elements $b_0,\dots,b_{n-1}$ such
 that 
 \begin{equation}\big|\bigcap_{j\in n\backslash \{j_0\}} Rb_jR:
 \bigcap_{j\in n} Rb_jR\big|\geq m\tag{*}\end{equation}
 for any $j_0<n$ (where $n=\{0,1,\dots,n-1\})$.
 So fix any  $n,m<\omega$, and for $N=nm$ choose $i_j$ and $r_j$ as in the claim.
 Let $(i_{j,k})_{j<n,j<m}$ be another enumeration of $(i_j)_{j<N}$, and let
 $(r_{j,k})_{j<n,k<m}$ be the corresponding enumeration of $(r_j)_{j<N}$
 and  $(\pi_{j,k})_{j<n,k<m}$ the corresponding enumeration of $(\pi_j)_{j<N}$.
 For any $j_0<n,k_0<m$ let $s_{j_0,k_0}\in R$ be such that $\pi_{j,k}(s_{j_0,k_0})=
 0$ for $(j,k)\neq (j_0,k_0)$ and $\pi_{j_0,k_0}(s_{j_0,k_0})=
 r_{j_0,k_0}$.
 Put $b_j=\sum_{j'\neq j, k<m}s_{j',k}$ for all $j<n$.
 \begin{claim}
  $|\bigcap_{j\in n\backslash \{j_0\}} Rb_jR:
 \bigcap_{j\in n} Rb_jR|\geq m$ for any $j_0<n$.
 \end{claim}
\begin{proof}
Fix any $j_0<n$ and put $b=b_0b_1\dots b_{j_0-1}b_{j_0+1}b_{j_0+2}\dots b_{n-1}$.
Notice that for any $r\in \bigcap_{j\in n} Rb_{j}R$
and $k<m$ we have that $\pi_{j_0,k}(r)=0$. On the other hand, for distinct 
$k_1,k_2<m$ we have that 
$$\pi_{j_0,k_1}(s_{j_0,k_1}b-s_{j_0,k_2}b)=
\pi_{j_0,k_1}(s_{j_0,k_1}b)=\pi_{j_0,k_1}(s_{j_0,k_1})\pi_{j_0,k_1}(b)=
r_{j_0,k_1}r_{j_0,k_1}^{n-1}=r_{j_0,k_1}^{n}\neq 0.$$ 
Hence the elements 
$$s_{j_0,0}b,s_{j_0,1}b,\dots,s_{j_0,m-1}b\in 
\bigcap_{j\in n\backslash \{j_0\}}Rb_jR$$ are in pairwise distinct cosets
of $\bigcap_{j\in n} Rb_jR$.  
\end{proof}
By the claim and $(*)$ we obtain a contradiction.
\end{proof}

\end{document}